\documentclass[11pt]{amsart}
\usepackage[T1]{fontenc}
\usepackage{lmodern}
\usepackage[margin=1.05in]{geometry}
\usepackage{amsmath,amssymb,mathtools}
\usepackage{microtype}
\numberwithin{equation}{section}
\usepackage[colorlinks=true,linkcolor=blue,citecolor=blue,urlcolor=blue]{hyperref}
\newtheorem{theorem}{Theorem}[section]
\newtheorem{proposition}[theorem]{Proposition}
\newtheorem{lemma}[theorem]{Lemma}
\newtheorem{corollary}[theorem]{Corollary}
\theoremstyle{definition}
\newtheorem{definition}[theorem]{Definition}
\theoremstyle{remark}
\newtheorem{remark}[theorem]{Remark}
\newcommand{\N}{\mathbb N}
\newcommand{\R}{\mathbb R}
\newcommand{\Baire}{\N^{\N}}
\newcommand{\T}{\N^{<\N}}
\newcommand{\Einf}{E_{\ell_\infty}}
\newcommand{\red}{\leq_B}
\newcommand{\bired}{\sim_B}
\newcommand{\cat}{\mathbin{{}^\frown}}
\newcommand{\restr}{\mathbin{\upharpoonright}}
\newcommand{\norm}[1]{\left\lVert #1\right\rVert}
\newcommand{\abs}[1]{\left|#1\right|}
\DeclareMathOperator{\diam}{diam}
\DeclareMathOperator{\spanR}{span_{\R}}

\title[No universal $F_\sigma$ equivalence relation]
{There is no universal $F_\sigma$ equivalence relation}
\author{Ruijun Wang}
\date{}

\begin{document}
\begin{abstract}
Rosendal proved that any $K_\sigma$ equivalence relation is Borel reducible to $\ell_\infty$. We prove that for every $F_\sigma$ equivalence relation $F$ on a Polish space, there is an $F_\sigma$ real linear subspace $H$ of $\ell_2$ such that the coset relation $x\mathrel{E_H}y$ if and only if $x-y\in H$ has no Baire-measurable reduction to $F$. In particular, there is no $F_\sigma$ equivalence relation universal under Borel reducibility for all $F_\sigma$ equivalence relations. This answers a question of Rosendal.
\end{abstract}
\maketitle

\section{Introduction}\label{sec:introduction}

A classification problem asks the complexity between two mathematical objects. Invariant descriptive set theory studies such problems by representing the objects as equivalence relations on Polish spaces. Its basic comparison is \emph{Borel reducibility}.
For equivalence relations $E$ and $F$, a \emph{Borel reduction} from $E$ to $F$, denoted by $E\leq_B F$,  is a Borel function $f:X\to Y$ such that $$xEy\Longleftrightarrow f(x)Ff(y).$$ We say that $E$ is Borel reducible to $F$. Intuitively, $F$ is more complex than $E$. For example, the orbit equivalence relation of an irrational rotation on a unit circle is not Borel reducible to the equality, so it does not admit a real number as an invariant.

An influential step in the development was the work
of Friedman and Stanley \cite{FS89}, who developed a reducibility theory
for classes of countable structures and compared natural classes of
algebraic structures. Another landmark was the
Harrington--Kechris--Louveau dichotomy \cite{HKL90}: a Borel equivalence
relation either admits classification by real numbers, called \emph{smooth} equivelance relations, or is at least as complicated as eventual equality
of binary sequences, denoted by $E_0$. This theorem identifies a fundamental boundary
between classification problems that admit simple complete invariants
and those that do not.

The subject also developed through the study of Polish group actions.
Becker and Kechris \cite{BK96} gave a systematic treatment of their
descriptive set theory and orbit equivalence relations. Hjorth's
turbulence theory \cite{Hjo00} supplied a major method for showing that
certain orbit relations cannot be classified by countable structures.
These developments connected topology and dynamics with logic and the
comparison of classification problems.

We refer readers to Gao's monograph \emph{Invariant Descriptive Set Theory} \cite{Gao09}
for this subject of research. An equivalence relation is usually classified by several benchmark equivalence relations, see \cite[Figure 15.1]{Gao09} and \cite[Figure 15.2]{Gao09} for concrete benchmark equivalence relations. In particular, \cite[Section 8.4]{Gao09} talks about examples of $F_\sigma$ equivalence relations and the question considered in this paper.

An equivalence relation $E$ is called a universal $\mathcal E$ equivalence relation if any equivalence relation in $\mathcal E$ is Borel reducible to $E$. For example, $E_\infty$ is a universal countable Borel equivalence relation. A benchmark of interest is \emph{bounded difference}, denoted by $E_{\ell_\infty}$, or $\ell_\infty$ for short,  on the space
\(\R^{\N}\) of real sequences:
\[
 u\mathrel{\Einf}v
 \quad\Longleftrightarrow\quad
 \sup_{i\in\N}\abs{u(i)-v(i)}<\infty.
\]
Rosendal \cite[Proposition~19]{Ros05} proved that \(\Einf\) is Borel
bireducible with a universal \(K_\sigma\) equivalence relation.
Consequently, every \(K_\sigma\) equivalence relation is Borel reducible
to \(\Einf\). Here \(K_\sigma\) means a countable union of compact sets,
whereas \(F_\sigma\) means a countable union of closed sets.
Rosendal asked whether every \(F_\sigma\) equivalence relation also
reduces to this benchmark \cite[Problem~28]{Ros05}.
Gao gives the universality theorem as
\cite[Theorem~8.4.2]{Gao09} and records the question on page~196.
It is also recorded as an open problem in Lecomte
\cite{Lec20} and in Question~1 of Rossegger, Slaman, and Steifer
\cite{RSS25}.

We answer Rosendal's question negatively.

\begin{theorem}\label{thm:main}
For every $F_\sigma$ equivalence relation $F$ on a Polish space, there
is an $F_\sigma$ real linear subspace $H$ of the real Hilbert space
$\ell_2$ such that the $F_\sigma$ equivalence relation
\[
 x\mathrel{E_H}y\quad\Longleftrightarrow\quad x-y\in H
\]
has no Baire-measurable reduction to $F$. In particular, there is no
universal $F_\sigma$ equivalence relation under Borel reducibility.
\end{theorem}

This paper is a collaboration with artificial intelligence. The AI model is gpt 6. The AI proves the main theorem under human suggestions and writes the main body of early drafts. The author writes the abstract and induction and improves readability. The author has reviewed all AI-generated content.

The rest of the paper is organized as follows. Section~\ref{sec:preliminaries} supplies the definitions and elementary
facts used in the proof. Sections~\ref{sec:target} and
\ref{sec:escape} define the two derivatives and establish their
comparison. Section~\ref{sec:construction} constructs the initial
branch family. Section~\ref{sec:tensors} proves the tensor span
lemma, Section~\ref{sec:transfinite} performs the ordinal recursion,
and Section~\ref{sec:conclusion} finishes the argument. The proof
does not invoke turbulence, effective descriptive set theory, or
Polishability of the subgroups.

\section{Preliminaries}\label{sec:preliminaries}

\subsection{Polish spaces and definable sets}

A \emph{Polish space} is a topological space whose topology is induced
by a complete separable metric. The real line, separable Hilbert spaces,
and countable products of Polish spaces are examples. Closed subsets and
\(G_\delta\) subsets of Polish spaces are Polish in their subspace topologies. All occurrences
of \(\R^{\N}\) below refer to the product topology. For instance, it is
induced by the complete metric
\[
 d(u,v)=\sum_{i=0}^{\infty}2^{-i-1}
              \min\{1,\abs{u(i)-v(i)}\}.
\]
Convergence in this topology is exactly convergence in each coordinate.

The \emph{Borel sets} of a topological space form the smallest
\(\sigma\)-algebra containing its open sets. A map is \emph{Borel} if
the inverse image of every open set is Borel. A set is \(F_\sigma\)
if it is a countable union of closed sets, and \(G_\delta\) if it is
a countable intersection of open sets. It is \(K_\sigma\) if it is a
countable union of compact sets. In a metrizable space every compact
set is closed, so every \(K_\sigma\) set is \(F_\sigma\).

An equivalence relation on a Polish space \(X\) is called Borel,
\(F_\sigma\), or \(K_\sigma\) according to its complexity as a subset
of \(X\times X\). These latter two descriptions refer to the specified
topology on \(X\). For further background on these definitions and
their role in classification, see \cite{Gao09}.

\subsection{Borel reducibility and coset relations}

\begin{definition}
Let \(E\) and \(F\) be equivalence relations on Polish spaces \(X\)
and \(Y\). A \emph{Borel reduction} of \(E\) to \(F\) is a Borel map
\(f:X\to Y\) such that, for all \(x,x'\in X\),
\[
 x\mathrel{E}x'
 \quad\Longleftrightarrow\quad
 f(x)\mathrel{F}f(x').
\]
We write \(E\red F\) if such a map exists, and \(E\bired F\) if
both \(E\red F\) and \(F\red E\).
A relation \(F\) is \emph{universal for a class of equivalence relations}
if every member of that class is Borel reducible to \(F\).
\end{definition}

A reduction need not be injective on points; it does induce an
injection on equivalence classes. Borel reducibility is transitive,
because compositions of Borel maps are Borel.

An equivalence relation is \emph{smooth} if it is Borel reducible to
equality on \(\R\). This says that real numbers can be assigned as
complete invariants by a Borel map. A basic nonsmooth benchmark is
eventual equality \(E_0\) on \(\{0,1\}^{\N}\):
\[
 x\mathrel{E_0}y
 \quad\Longleftrightarrow\quad
 (\exists m)(\forall n\ge m)\ x(n)=y(n).
\]
The dichotomy mentioned in the introduction says that every Borel
equivalence relation is either smooth or has \(E_0\) Borel reducible
to it \cite{HKL90}. We will not use this theorem in the proof.

If a group \(G\) acts on \(X\), its \emph{orbit equivalence relation}
identifies \(x\) and \(y\) when \(g\cdot x=y\) for some \(g\in G\).
A particularly simple algebraic example is a coset relation.
For an additive subgroup \(H\) of a real Banach space \(B\), define
\[
 b\mathrel{E_H}c\quad\Longleftrightarrow\quad b-c\in H.
\]
The subgroup axioms give reflexivity, symmetry, and transitivity.
If \(H\) is \(F_\sigma\) in \(B\), then \(E_H\) is \(F_\sigma\) in
\(B\times B\), since subtraction is continuous. We do not assume
that \(H\), with its subspace topology, is a Polish group.

The space
\[
 \ell_\infty=\{u\in\R^{\N}:\sup_i\abs{u(i)}<\infty\}
\]
is an additive subgroup of \(\R^{\N}\); its coset relation is
\(\Einf\). We use no norm topology on this subgroup. Define the
extended distance
\[
 \Delta(u,v)=\sup_{i\in\N}\abs{u(i)-v(i)}\in[0,\infty].
\]
It satisfies the triangle inequality, and
\(u\mathrel{\Einf}v\) exactly when \(\Delta(u,v)<\infty\).
For each integer \(M\ge1\), the set
\[
 R_M=\{(u,v):\Delta(u,v)\le M\}
     =\bigcap_{i\in\N}\{(u,v):\abs{u(i)-v(i)}\le M\}
\]
is closed in the product topology. Thus \(\Einf=\bigcup_{M\ge1}R_M\)
is \(F_\sigma\). If \(u_n\to u\) coordinatewise and
\(\Delta(u_n,v)\le M\) for all \(n\), then \(\Delta(u,v)\le M\),
by passing to the limit in each coordinate.

\begin{theorem}[Rosendal {\cite[Proposition~19]{Ros05}}]\label{thm:rosendal}
Every \(K_\sigma\) equivalence relation on a Polish space is Borel
reducible to \(\Einf\). Moreover, \(\Einf\) is Borel bireducible
with a \(K_\sigma\) equivalence relation.
\end{theorem}

\begin{remark}\label{rem:ksigma}
Being universal for \(K_\sigma\) relations does not assert that the
given presentation of \(\Einf\) is itself \(K_\sigma\).
In fact it is not \(K_\sigma\) as a subset of
\((\R^{\N})^2\). Otherwise its intersection with the closed diagonal
would be \(K_\sigma\), making \(\R^{\N}\) a countable union of compact
sets. To see that this is impossible, given compact sets \(K_n\subseteq
\R^{\N}\), choose \(x(n)\) outside the compact set
\(\{u(n):u\in K_n\}\). Then \(x\notin\bigcup_nK_n\).
Theorem~\ref{thm:rosendal} is a statement about Borel reducibility
and Borel bireducibility.
\end{remark}

\subsection{Category and continuity on a comeager set}

A set is \emph{nowhere dense} if its closure has empty interior,
and \emph{meager} if it is a countable union of nowhere dense sets.
Its complement is then \emph{comeager}. A space is a \emph{Baire space}
if every countable intersection of dense open sets is dense.
The Baire category theorem states that complete metric spaces are
Baire spaces. In particular, a nonempty open subset of a Polish space
cannot be meager in that space.

A set has the \emph{Baire property} if it differs from an open set
by a meager set. Sets with the Baire property form a \(\sigma\)-algebra
containing the open sets. For countable unions this follows directly
from the definition. For complements, the complement of an open set
differs from its interior by a subset of the boundary of that open
set, which is nowhere dense. Consequently every Borel set has the
Baire property. A map is \emph{Baire measurable} if inverse images
of open sets have the Baire property; every Borel map is Baire
measurable. A Baire-measurable reduction is defined by the same
equivalence-preserving and equivalence-reflecting condition as a Borel
reduction, with the map required only to be Baire measurable.

\begin{lemma}\label{lem:restriction}
Let \(X\) be a Polish space and \(Y\) a second-countable space.
Every Baire-measurable map \(f:X\to Y\) has a continuous restriction
to some comeager \(G_\delta\) subset \(C\subseteq X\).
\end{lemma}

\begin{proof}
Let \((U_n)_{n\in\N}\) be a countable base for \(Y\). Choose open
\(V_n\subseteq X\) such that
\(f^{-1}(U_n)\mathbin{\triangle}V_n\) is meager.
The union of these errors is contained in a meager \(F_\sigma\)
set \(F\): take a countable cover by nowhere dense sets and replace
each by its closure. Set \(C=X\setminus F\). For every \(n\),
\[
 (f|_C)^{-1}(U_n)=C\cap V_n,
\]
which is open in \(C\). These identities prove continuity of the
restriction.
\end{proof}

The conclusion concerns the relative topology on \(C\). It does not
assert that the original map is continuous at each point of \(C\)
when approached through arbitrary points of \(X\).

\begin{lemma}\label{lem:category}
The following facts will be used.
\begin{enumerate}
\item A dense \(G_\delta\) subspace of a Baire space is a Baire space.
\item If \(p:X\to Y\) is continuous and open, the inverse image of
a dense open subset of \(Y\) is dense open in \(X\).
Consequently, inverse images of comeager \(G_\delta\) sets under
such a map are comeager \(G_\delta\).
\item Translation by a fixed vector in a Banach space preserves
comeager sets.
\end{enumerate}
\end{lemma}

\begin{proof}
For (1), write \(D=\bigcap_nG_n\), where each \(G_n\) is open and
dense in the ambient Baire space \(X\). Write any sequence of
relatively open dense subsets of \(D\) as \(D\cap V_n\), with
\(V_n\) open in \(X\). Each \(V_n\) is dense in \(X\), since \(D\)
is dense. The Baire property of \(X\), applied to all \(G_n\) and
\(V_n\), shows that their intersection is dense in \(X\), and hence
dense in \(D\).

For (2), continuity gives openness of the inverse image. If
\(W\subseteq X\) is nonempty open, then \(p(W)\) is nonempty open,
so it meets the prescribed dense open set. This proves density.
Equivalently, inverse images of closed nowhere dense sets are closed
nowhere dense. Since every meager set is contained in a countable
union of closed nowhere dense sets, inverse images of meager sets
are meager. Continuity also preserves the \(G_\delta\) property
under inverse images, giving the stated consequence.
Finally, (3) follows because translations are homeomorphisms.
\end{proof}

\subsection{Branches and Hilbert coordinates}

We use \(\N=\{0,1,2,\ldots\}\). The \emph{Baire space}
\(P=\Baire\) is the product of countably many copies of the
discrete space \(\N\). It is Polish: a complete compatible metric
sets the distance between distinct \(x,y\) equal to \(2^{-m}\),
where \(m\) is their first differing coordinate; eventually zero
sequences form a countable dense set.

Let \(T=\T\) be the set of finite sequences of natural numbers.
For \(s\in T\), write \(\abs{s}\) for its length and \(s\cat k\)
for the sequence obtained by appending \(k\).
For \(x\in P\), \(x\restr L\) is its prefix of length \(L\).
The same notation applies to initial segments of a finite sequence.
We say that \(x\) or a finite sequence \(t\) \emph{extends} \(s\)
if its first \(\abs{s}\) entries are those of \(s\).
The cylinders
\[
 [s]=\{x\in P:x\text{ extends }s\}\qquad(s\in T)
\]
form a countable base of nonempty sets that are both open and closed.

For a countable index set \(I\), the real Hilbert space \(\ell_2(I)\)
consists of families \(b=(b_i)_{i\in I}\) with
\(\sum_{i\in I}b_i^2<\infty\), equipped with the corresponding norm.
A subset of the standard orthonormal basis determines a
\emph{closed coordinate subspace}: its closed linear span.
The orthogonal projection onto that subspace simply keeps the
specified coordinates and sets all others to zero.
These projections have norm at most one. A vector has
\emph{finite coordinate support} if only finitely many of its
coordinates are nonzero.

\subsection{Countable ordinals and decreasing derivatives}

We use ordinals to index repeated removal operations. The ordinal
$\omega$ is the order type of the natural numbers, and $\omega_1$ is
the first uncountable ordinal. An ordinal $\alpha+1$ is a successor;
a nonzero ordinal with no immediate predecessor is a limit. Every
countable limit ordinal $\lambda$ has an increasing cofinal sequence
$(\alpha_n)$: each $\alpha_n<\lambda$, and every $\xi<\lambda$ is at
most some $\alpha_n$. A countable supremum of countable ordinals is
again countable.

If an operation $D$ sends a closed set to a closed subset of itself,
its transfinite iteration starts from $A^0=A$, takes
$A^{\alpha+1}=D(A^\alpha)$ at successors, and takes
$A^\lambda=\bigcap_{\alpha<\lambda}A^\alpha$ at limits. These are
decreasing closed sets. Transfinite induction proves an assertion
by treating stage zero, showing preservation at a successor, and
showing preservation at a limit. Our constructions use only
countable stages. The rank terminology below refers to the number
of iterations needed to remove points; no effective coding of
ordinals is used.

\subsection{Hilbert direct sums and tensor products}

All vector spaces in the construction are real. If $(B_i)$ is a
countable family of Hilbert spaces, its Hilbert direct sum consists
of tuples $(v_i)$ with $\sum_i\|v_i\|^2<\infty$, with squared norm
equal to that sum. Coordinate projections are bounded linear maps.
If all $B_i$ are separable, so is the direct sum. A real Hilbert
space with a countably infinite orthonormal basis is isometric to
$\ell_2$, the space of square-summable real sequences.

For real Hilbert spaces $U,V$, their Hilbert tensor product
$U\widehat\otimes V$ is the completion of their algebraic tensor
product for the inner product determined by
\[
 \langle u\otimes v,u'\otimes v'\rangle
 =\langle u,u'\rangle\langle v,v'\rangle.
\]
In particular $\|u\otimes v\|=\|u\|\|v\|$. Given orthonormal bases
$(e_j)$ and $(f_i)$, the vectors $e_j\otimes f_i$ form an orthonormal
basis of the completed tensor product. Thus its elements can be
viewed as square-summable matrices $(t_{ji})$. Such a matrix defines
a bounded operator $T:V\to U$ by its columns. These operators are
called Hilbert--Schmidt operators, and their Hilbert--Schmidt norm
is $(\sum_{j,i}|t_{ji}|^2)^{1/2}$. Under this identification,
$u\otimes v$ is $x\mapsto\langle x,v\rangle u$.

We will use several elementary consequences. The operator maps
$T\mapsto Tf_i$ and $T\mapsto T^*e_j$ are continuous in this norm.
Operators of rank at most $r$ form a closed set: all their
$(r+1)\times(r+1)$ coordinate minors vanish, and conversely vanishing
of these minors forces rank at most $r$. Bounded linear functionals
can be applied to selected tensor factors; these contractions are
bounded linear maps on the tensor product. Finite Hilbert tensor
products are associative, and distribute over finite or countable
Hilbert direct sums by matching the orthonormal bases. These facts
justify the finite packet projections and contractions used below.

A \emph{closed embedding} is a continuous injective map that is a
homeomorphism onto a closed subset of its target. A real-valued
function is \emph{locally constant} if it is constant on some
neighborhood of each point. Each nonempty fiber of a locally
constant function is open, so such a function on a second countable
space has countable range. Its fibers are also closed. These
observations will make the tensor closed cover countable.

\section{A composition rank for a fixed target}\label{sec:target}

Let $F$ be an $F_\sigma$ equivalence relation on a nonempty Polish space
$Y$. Fix an increasing sequence $(F_n)_{n\geq1}$ of closed reflexive
symmetric subsets of $Y^2$ with union $F$. Such a sequence exists:
starting with any closed cover, take finite unions, add their inverses,
and add the diagonal. No $F_n$ is assumed transitive.

For each $n$ put
\[
 Z_n=\{(a,b,c)\in Y^3:(a,b)\in F_n,\ (b,c)\in F_n\}.
\]
This is a closed Polish subspace of $Y^3$. Transitivity of $F$ makes
the following function finite everywhere on $Z_n$:
\[
 m_n(a,b,c)=\min\{k\geq1:(a,c)\in F_k\}.
\]
Its sublevel sets are closed. For a closed subset $A$ of $Z_n$ define
\[
 \partial_n A
 = A\setminus\bigcup_{k\geq1}
       \operatorname{Int}_A\bigl(A\cap\{m_n\leq k\}\bigr).
\]
Thus $\partial_n A$ consists of the points at which $m_n$ is unbounded
on every relative neighborhood in $A$. It is closed. Define
\[
 Z_n^0=Z_n,\qquad Z_n^{\xi+1}=\partial_n Z_n^\xi,
 \qquad Z_n^\lambda=\bigcap_{\xi<\lambda}Z_n^\xi
 \quad(\lambda\text{ a limit ordinal}).
\]

\begin{lemma}\label{lem:target-countable}
For each $n$ there is a countable ordinal $r_n$ such that $Z_n^{r_n}$
is empty. Consequently $\gamma_F=\sup_{n\geq1}r_n$ is countable.
\end{lemma}

\begin{proof}
If $A$ is a nonempty closed subset of $Z_n$, its closed subsets
$A\cap\{m_n\leq k\}$ cover it. By the Baire category theorem, at least
one has nonempty relative interior. Hence $\partial_n A\ne A$.
In a second countable space a decreasing sequence of closed sets can
have only countably many strict successor decreases. Indeed, for each
such decrease choose a basic open set meeting the earlier set and
disjoint from the later one; that basic open set cannot witness any
subsequent decrease. If all $Z_n^\xi$ for $\xi<\omega_1$ were nonempty,
every successor step would be strict, a contradiction. The supremum of
countably many countable ordinals is countable.
\end{proof}

We use only the rank of this fixed closed cover. We do not need to
minimize over covers or topologies, nor to assign an invariant rank to
the Borel reducibility degree of $F$.

\section{Escape derivatives and reductions}\label{sec:escape}

Let $H$ be an additive subgroup of a separable real Banach space $B$,
let $P$ be a nonempty Polish space, and let $j:P\to H\subseteq B$ be
continuous. Write $\sigma(t,u)=j(t)+j(u)$. For a closed $R\subseteq P^2$
define
\[
 \mathcal D_{j,H}(R)=
 \left\{p\in R:\text{for every open }U\ni p,
       \ \overline{\sigma(R\cap U)}^{\,B}\not\subseteq H\right\}.
\]
In words, a point survives if arbitrarily near it there are sums with
an ambient limit outside $H$. This is a closed subset of $R$: a
neighborhood whose image closure is contained in $H$ witnesses removal
of every point of $R$ in that neighborhood. Put
\[
 R_{j,H}^0=P^2,\qquad
 R_{j,H}^{\xi+1}=\mathcal D_{j,H}(R_{j,H}^\xi),\qquad
 R_{j,H}^\lambda=\bigcap_{\xi<\lambda}R_{j,H}^\xi.
\]
We often abbreviate this to $R^\xi$ when the family and subgroup are
fixed. All closures of sums are in the ambient Banach space, not in
the subgroup.

\begin{lemma}[Locality and transfer]\label{lem:escape-tools}
The following facts hold.
\begin{enumerate}
\item If $O\subseteq P$ is open, the derivatives of the restricted
family $j|_O$, computed on $O^2$, are $R_{j,H}^\xi\cap O^2$.
\item Suppose $i:Q\to K\subseteq C$ is another such family. Let
$a,b:Q\to P$ be continuous. Suppose that a vector $v\in B$ and a
bounded linear map $T:C\to B$ satisfy
\[
 j(a(t))+j(b(u))=v+T(i(t)+i(u))\quad(t,u\in Q),
\]
and that $v+T(z)\notin H$ whenever $z\notin K$. Then
\[
 (a\times b)(R_{i,K}^\xi)\subseteq R_{j,H}^\xi
 \quad\text{for every countable }\xi.
\]
\end{enumerate}
\end{lemma}

\begin{proof}
The first assertion follows by transfinite induction: at a point of
$O^2$ all neighborhoods in the definition may be shrunk to lie in
$O^2$. For the second, the assertion is immediate at stage zero and
is preserved by intersections at limits. At a successor stage, pull
an arbitrary target neighborhood back by $a\times b$. The source
derivative provides a sequence of pairs in the preceding source
remainder whose sums converge to some $z\notin K$. By induction the
image pairs lie in the preceding target remainder. Their sums converge
to $v+T(z)\notin H$, proving the successor assertion.
\end{proof}

For comparison, when $H$ is $F_\sigma$, every such escape derivative
also terminates at a countable stage. If $H=\bigcup_k C_k$ with $C_k$
closed and $R$ is nonempty closed, the closed sets
$R\cap\sigma^{-1}(C_k)$ cover $R$. Baire category gives a nonempty
relative open part whose image has closure in one $C_k$. Thus the
derivative removes a nonempty part. The countable-base argument of
Lemma~\ref{lem:target-countable} applies. The construction below will
make these countable ranks arbitrarily large.

\begin{lemma}[A rectangle with uniformly bounded edges]
\label{lem:rectangle}
Suppose $f:B\to Y$ is a Baire-measurable reduction of $E_H$ to the
relation $F$ of Section~\ref{sec:target}. Let $C\subseteq B$ be a
comeager $G_\delta$ set on which $f$ is continuous. There exist an
integer $n\geq1$ and nonempty open sets $V\subseteq B$, $O\subseteq P$
such that, for $x\in V$ and $t\in O$,
\begin{align*}
 x,x+j(t)\in C&\ \Longrightarrow\ (f(x+j(t)),f(x))\in F_n,\\
 x,x-j(t)\in C&\ \Longrightarrow\ (f(x),f(x-j(t)))\in F_n.
\end{align*}
Each implication requires only the two displayed endpoints in $C$.
\end{lemma}

\begin{proof}
The maps $(x,t)\mapsto x$, $(x,t)\mapsto x+j(t)$ and
$(x,t)\mapsto x-j(t)$ are continuous and open. Consequently
\[
 D=\{(x,t):x,x+j(t),x-j(t)\in C\}
\]
is a dense $G_\delta$ subset of $B\times P$, and is Polish. On $D$,
the simultaneous conditions
\[
 (f(x+j(t)),f(x))\in F_n,\qquad
 (f(x),f(x-j(t)))\in F_n
\]
define relatively closed sets. They cover $D$, since $j(t)\in H$ and
$f$ is a reduction. Baire category gives an $n$ and an open rectangle
$V\times O$ whose intersection with $D$ satisfies both conditions.
For the first implication at a point where only $x,x+j(t)\in C$ are
known, approximate $(x,t)$ by a sequence from $D\cap(V\times O)$.
The two relevant endpoints lie in $C$ throughout and converge to
points of $C$. Relative continuity of $f|_C$ and closedness of $F_n$
give the implication. The second is identical.
\end{proof}

\begin{proposition}[Comparison of the two derivatives]
\label{prop:comparison}
In the setting of Lemma~\ref{lem:rectangle}, for every countable ordinal
$\xi$, every $(t,u)\in R_{j,H}^\xi\cap O^2$, and every $x\in V$ such
that $x+j(t),x,x-j(u)\in C$, one has
\[
 \bigl(f(x+j(t)),f(x),f(x-j(u))\bigr)\in Z_n^\xi.
\]
\end{proposition}

\begin{proof}
We use transfinite induction. Stage zero is precisely the two edge
conditions of Lemma~\ref{lem:rectangle}. Limit stages follow by
intersection. Suppose the statement holds at stage $\xi$, and fix a
pair $(t,u)\in R^{\xi+1}\cap O^2$ and an $x$ as in the statement.
Write
\[
 p=(f(x+j(t)),f(x),f(x-j(u))).
\]
By induction $p\in Z_n^\xi$. If $p\notin Z_n^{\xi+1}$, there are a
neighborhood $W$ of $p$ in $Y^3$ and an integer $k$ such that
\begin{equation}
 Z_n^\xi\cap W\subseteq\{m_n\leq k\}.
\label{eq:local-composition-bound}
\end{equation}
Choose a product of three open subsets of $Y$ containing $p$ and
contained in $W$. Relative continuity on $C$ gives open neighborhoods
$N_+,N_0,N_-$ of $x+j(t),x,x-j(u)$ in $B$ whose intersections with $C$
map into the three chosen target sets; also require $N_0\subseteq V$.

Continuity of $j$ now gives open neighborhoods $U_t,U_u\subseteq O$
of $t,u$ and a nonempty open neighborhood $A$ of $x+j(t)$ such that
for every $v\in A$, $s\in U_t$, and $r\in U_u$,
\begin{equation}
 v\in N_+,\qquad v-j(s)\in N_0,\qquad
 v-j(s)-j(r)\in N_-.
\label{eq:source-neighborhoods}
\end{equation}
This follows directly by taking sufficiently small norm balls around
the three original source points.

Because $(t,u)\in R^{\xi+1}$, there are pairs
$(s_l,r_l)\in R^\xi\cap(U_t\times U_u)$ and a vector $z\notin H$ with
\[
 j(s_l)+j(r_l)\longrightarrow z.
\]
The set
\[
 A\cap C\cap(C+z)\cap
   \bigcap_l(C+j(s_l))\cap
   \bigcap_l(C+j(s_l)+j(r_l))
\]
is comeager in the nonempty open set $A$. Choose $v$ in it and set
$x_l=v-j(s_l)$. All three points
\[
 v=x_l+j(s_l),\qquad x_l,\qquad
 v-j(s_l)-j(r_l)=x_l-j(r_l)
\]
belong to $C$, and $x_l\in V$. Induction puts the corresponding target
triple in $Z_n^\xi$, while \eqref{eq:source-neighborhoods} puts it in
$W$. By \eqref{eq:local-composition-bound},
\[
 (f(v),f(v-j(s_l)-j(r_l)))\in F_k\quad\text{for every }l.
\]
The last source points converge to $v-z\in C$. Continuity of $f|_C$
and closedness of $F_k$ imply $(f(v),f(v-z))\in F_k\subseteq F$.
Since $f$ is a reduction, this gives $z\in H$, a contradiction.
Therefore $p\in Z_n^{\xi+1}$, completing the induction.
\end{proof}

The fixed first endpoint $v$ in this proof is chosen after the escape
sequence has been fixed. Thus only countably many translates of $C$
are intersected. No continuity of $f$ on an affine subspace is assumed.

\begin{corollary}\label{cor:rank-obstruction}
Suppose every diagonal point $(t,t)$ belongs to $R_{j,H}^\xi$. If
$E_H$ has a Baire-measurable reduction to $F$, then $Z_n^\xi$ is
nonempty for some $n$.
\end{corollary}

\begin{proof}
Use Lemma~\ref{lem:rectangle} and fix $t\in O$. The set
$V\cap C\cap(C-j(t))\cap(C+j(t))$ is nonempty by Baire category.
Any $x$ in it supplies a triple in $Z_n^\xi$ by
Proposition~\ref{prop:comparison}.
\end{proof}

\section{A linear subspace generated by branches}\label{sec:construction}

\subsection{Definition and tail estimates}

Let \(B\) be the real Hilbert space with orthonormal basis
\[
 \{h_s:s\in T\}\ \cup\ \{e_{s,n}:s\in T,\ n\in\N\},
\]
all of whose displayed vectors are distinct.
This basis is countably infinite, so \(B\) is isometric to \(\ell_2\).
A coordinate \(h_s\) or \(e_{s,n}\) is said to be attached to the node
\(s\).

Define positive numbers \(r_s\), for \(s\in T\), recursively by
\begin{equation}\label{eq:weights}
 r_{\varnothing}=1,\qquad
 r_{s\cat k}=2^{-(k+2)}r_s.
\end{equation}
In particular \(r_{s\cat k}\le r_s/4\), and
\(r_s\le4^{-\abs{s}}\). For \(x\in P=\Baire\), set
\begin{equation}\label{eq:branch}
 j(x)=\sum_{\ell=0}^{\infty}r_{x\restr\ell}
 \left(h_{x\restr\ell}
       +(-1)^{x(\ell)}
        e_{x\restr\ell,\lfloor x(\ell)/2\rfloor}\right).
\end{equation}
The terms at distinct levels have disjoint coordinate supports.

For \(s\in T\), define the finite prefix vector
\[
 v_s=\sum_{\ell<\abs{s}}r_{s\restr\ell}
 \left(h_{s\restr\ell}
       +(-1)^{s(\ell)}
        e_{s\restr\ell,\lfloor s(\ell)/2\rfloor}\right),
\]
where \(v_{\varnothing}=0\).
For \(x\in[s]\), put \(J_s(x)=j(x)-v_s\).

\begin{lemma}\label{lem:tailbound}
The series \eqref{eq:branch} converges in \(B\).
For \(s\in T\) and \(x\in[s]\),
\[
 \norm{J_s(x)}^2\le\frac{32}{15}r_s^2,
 \qquad
 \norm{J_s(x)}<2r_s.
\]
The map \(j:P\to B\) is continuous. The vector \(J_s(x)\) is
supported on coordinates attached to nodes extending \(s\), and
its \(h_s\)-coordinate is \(r_s\).
\end{lemma}

\begin{proof}
Starting at \(s\), the squared norm of the term \(m\) levels later
is at most \(2r_s^2\,16^{-m}\). Orthogonality gives
\[
 \sum_{m=0}^{\infty}2r_s^2\,16^{-m}
 =\frac{32}{15}r_s^2<4r_s^2.
\]
This proves convergence and the tail estimates.
For \(x,y\in[s]\), the prefix vectors agree, so
\[
 \norm{j(x)-j(y)}
 \le\norm{J_s(x)}+\norm{J_s(y)}
 \le2\sqrt{32/15}\,r_s<4r_s.
\]
Consequently \(\diam j([s])<4r_s\le4^{1-\abs{s}}\), which
tends to zero with \(\abs{s}\). This proves continuity.
The support and the \(h_s\)-coordinate follow from the formula.
\end{proof}

Define the real linear subspace
\begin{equation}\label{eq:H}
 H=\spanR j(P).
\end{equation}
The span in \eqref{eq:H} is the algebraic span: its elements are
finite real linear combinations of branch vectors. No closure
is taken.

\subsection{Closed tail images and an \(F_\sigma\) decomposition}

\begin{lemma}\label{lem:closedtails}
For every \(s\in T\), the set \(J_s([s])\) is closed in \(B\).
\end{lemma}

\begin{proof}
Suppose \((J_s(x_m))_m\), with \(x_m\in[s]\), is a Cauchy
sequence. Project onto the closed coordinate subspace generated
by \(\{e_{s,n}:n\in\N\}\).
The projected vectors belong to
\[
 \{r_se_{s,n},-r_se_{s,n}:n\in\N\}.
\]
Any two distinct vectors in this set have distance at least
\(\sqrt{2}\,r_s\). Since the projections are Cauchy, they are
eventually constant. The coordinate index and sign determine
\(x_m(\abs{s})\) uniquely: index \(n\) with positive sign gives
digit \(2n\), and index \(n\) with negative sign gives digit
\(2n+1\).

After this digit stabilizes, subtract the fixed term attached to
\(s\) and repeat the argument at the next node. Inductively,
every digit eventually stabilizes, defining a branch \(x\in[s]\).
For each longer prefix \(t\) of \(x\), all sufficiently late
\(x_m\) belong to \([t]\). Their fixed initial terms agree with
those of \(x\), and Lemma~\ref{lem:tailbound} gives
\[
 \norm{J_s(x_m)-J_s(x)}<4r_t
\]
for all such \(m\). Since \(r_t\le4^{-\abs{t}}\), the sequence
converges to \(J_s(x)\).
In particular every convergent sequence in \(J_s([s])\) has
its limit in that set, proving closedness.
\end{proof}

\begin{proposition}\label{prop:fsigma}
The subspace \(H\) is \(F_\sigma\) in \(B\).
\end{proposition}

\begin{proof}
For an integer \(L\ge0\), a finite nonempty set \(S\subseteq\N^L\),
and an integer \(k\ge1\), define
\[
 C(S,k)=
 \left\{\sum_{s\in S}a_sj(x_s):
       x_s\in[s],\quad 1/k\le\abs{a_s}\le k\ (s\in S)\right\}.
\]
We first show that each \(C(S,k)\) is closed.
Suppose
\[
 b_m=\sum_{s\in S}a_{m,s}j(x_{m,s})\longrightarrow b,
 \qquad x_{m,s}\in[s],\quad 1/k\le\abs{a_{m,s}}\le k.
\]
The nodes in \(S\) are distinct and have the same length.
Only the branch extending \(s\) contributes to the \(h_s\)
coordinate. Hence
\[
 \langle b_m,h_s\rangle=r_sa_{m,s},
 \qquad
 a_{m,s}\longrightarrow a_s=\frac{\langle b,h_s\rangle}{r_s}.
\]
The limit satisfies \(1/k\le\abs{a_s}\le k\).

Let \(P_s\) be the orthogonal projection onto the closed coordinate
subspace generated by all basis vectors attached to nodes extending
\(s\). Distinct nodes in \(S\) have disjoint descendant coordinate
subspaces. The ancestors of each node occur only in its fixed
prefix vector. It follows that
\[
 P_sb_m=a_{m,s}J_s(x_{m,s}).
\]
The left side converges to \(P_sb\). Dividing by \(a_{m,s}\),
which converges to a nonzero number, gives
\[
 J_s(x_{m,s})\longrightarrow \frac{P_sb}{a_s}.
\]
Lemma~\ref{lem:closedtails} supplies \(x_s\in[s]\) for which
the limit is \(J_s(x_s)\).
Now pass to the limit in the full finite sum
\[
 b_m=\sum_{s\in S}a_{m,s}\bigl(v_s+J_s(x_{m,s})\bigr).
\]
This yields
\(b=\sum_{s\in S}a_s(v_s+J_s(x_s))
  =\sum_{s\in S}a_sj(x_s)\in C(S,k)\), proving closedness.

There are countably many choices of \(L,S,k\). Every finite linear
combination of branch vectors can be simplified by combining repeated
branches and removing zero coefficients. If the resulting combination
is nonempty, choose \(L\) so that the remaining branches have distinct
prefixes of length \(L\). Choose \(k\) large enough that every
coefficient has absolute value between \(1/k\) and \(k\).
The vector then belongs to the corresponding \(C(S,k)\).
Therefore
\[
 H=\{0\}\cup
   \bigcup_{\substack{L\ge0,\ S\subseteq\N^L\\
                     S\text{ finite, nonempty}}}
   \ \bigcup_{k\ge1}C(S,k),
\]
a countable union of closed sets.
\end{proof}

\subsection{Limits of paired branches outside the span}

\begin{lemma}\label{lem:finite}
The only vector of \(H\) with finite coordinate support is \(0\).
Moreover, the family \(j(P)\) is linearly independent.
\end{lemma}

\begin{proof}
Consider a nonempty finite combination
\(b=\sum_{i=1}^Na_i j(x_i)\) with distinct branches \(x_i\) and
nonzero coefficients \(a_i\).
For all sufficiently large \(L\), the prefixes
\(x_i\restr L\) are pairwise distinct. Thus the
\(h_{x_1\restr L}\)-coordinate of \(b\) is exactly
\(a_1r_{x_1\restr L}\ne0\).
These are infinitely many distinct coordinates as \(L\) varies.
Every such combination therefore has infinite coordinate support
and is nonzero. Both assertions follow.
\end{proof}

\begin{proposition}\label{prop:escape}
For every \(s\in T\),
\[
 \overline{j([s])+j([s])}\not\subseteq H.
\]
\end{proposition}

\begin{proof}
Let \(x_n=s\cat(2n)\cat 0^\infty\) and
\(y_n=s\cat(2n+1)\cat 0^\infty\), where the notation means
that all subsequent digits are zero.
The two vectors share the prefix vector \(v_s\).
At node \(s\), their \(e_{s,n}\)-coordinates have opposite signs.
Consequently
\[
 \begin{split}
 j(x_n)+j(y_n)
  ={}&2v_s+2r_sh_s\\
     &{}+J_{s\cat(2n)}(x_n)
       +J_{s\cat(2n+1)}(y_n).
 \end{split}
\]
The last two terms tend to zero in norm, since their norms are
bounded by twice their respective weights and
\[
 r_{s\cat(2n)}=2^{-(2n+2)}r_s\longrightarrow0,
 \qquad
 r_{s\cat(2n+1)}=2^{-(2n+3)}r_s\longrightarrow0.
\]
Thus \(j(x_n)+j(y_n)\to z_s=2v_s+2r_sh_s\).
The vector \(z_s\) has finite coordinate support and its
\(h_s\)-coordinate is \(2r_s>0\), so it is nonzero.
Lemma~\ref{lem:finite} implies \(z_s\notin H\).
Since \(x_n,y_n\in[s]\), this proves the assertion.
\end{proof}

\section{Countable tensor families with $F_\sigma$ spans}\label{sec:tensors}
The next lemma first handles finite algebraic tensor products.
The second supplies the countable construction needed at limit ordinals.
The notation $\widehat\otimes$ denotes a completed Hilbert tensor
product; $\otimes_{\mathrm{alg}}$ denotes finite sums of elementary
tensors, with no completion.

\begin{lemma}[Finite algebraic tensor products]\label{lem:algebraic-tensor-fsigma}
Let $U$ and $V$ be separable real Hilbert spaces, and let $H\subseteq U$
and $K\subseteq V$ be $F_\sigma$ linear subspaces. Then the algebraic
tensor product $H\otimes_{\mathrm{alg}}K$, regarded as a linear subspace
of the Hilbert tensor product $U\widehat\otimes V$, is $F_\sigma$.
Consequently the same assertion holds for any finite number of factors.
\end{lemma}

\begin{proof}
Identify $U\widehat\otimes V$ with the Hilbert space of Hilbert--Schmidt
operators $T:V\to U$, identifying $u\otimes v$ with the operator
$x\mapsto\langle x,v\rangle u$. Fix orthonormal bases $(e_j)$ of $U$
and $(f_i)$ of $V$, and closed covers
\[
 H=\bigcup_{a\geq 1} H_a,\qquad K=\bigcup_{b\geq 1}K_b.
\]
The bases and their index sets may be finite. For a positive integer
$r$, choose $r$ distinct column indices $I=(i_1,\ldots,i_r)$ and $r$
distinct row indices $J=(j_1,\ldots,j_r)$. Write
\[
 M_{J,I}(T)=\bigl(\langle Tf_{i_t},e_{j_s}\rangle\bigr)_{s,t=1}^r.
\]
Fix positive integers $k,a_1,\ldots,a_r,b_1,\ldots,b_r$ and impose the
conditions
\begin{gather*}
 \operatorname{rank}(T)\leq r,\qquad
 |\det M_{J,I}(T)|\geq k^{-1},\\
 Tf_{i_t}\in H_{a_t}\quad(1\leq t\leq r),\qquad
 T^*e_{j_s}\in K_{b_s}\quad(1\leq s\leq r).
\end{gather*}
These conditions define a closed subset of the Hilbert--Schmidt
operator space. Indeed, rank at most $r$ is equivalent to the vanishing
of every $(r+1)\times(r+1)$ coordinate minor, and all the displayed
coordinate maps are continuous.

Every operator in this closed set belongs to
$H\otimes_{\mathrm{alg}}K$. Indeed, let $C:\mathbb R^r\to U$ have
columns $Tf_{i_t}$, and let $R:V\to\mathbb R^r$ have coordinates
$x\mapsto\langle Tx,e_{j_s}\rangle$. The invertibility of $M=M_{J,I}(T)$
and the rank bound imply
\[
 T=CM^{-1}R.
\]
To verify this equality, the selected columns are linearly independent
and hence span the range of $T$; the selected rows determine their
coefficients by multiplication by $M$. The equality expresses $T$ as
a finite sum of tensors of selected columns from $H$ and selected rows
from $K$.

Conversely, if $T\in H\otimes_{\mathrm{alg}}K$ has positive rank $r$,
then every column $Tf_i$ lies in $H$ and every row $T^*e_j$ lies in $K$.
Some $r\times r$ coordinate minor is invertible. Choose this minor, a
positive lower bound $k^{-1}$ for its absolute determinant, and closed
cover members containing its finitely many selected rows and columns.
Thus $T$ belongs to one of the closed sets above. There are only
countably many choices of all this data. Adding the closed singleton
$\{0\}$ gives a countable closed cover of
$H\otimes_{\mathrm{alg}}K$. Iteration proves the finite-factor assertion.
\end{proof}

\begin{lemma}[A countable tensor construction with an $F_\sigma$ span]
\label{lem:fock-fsigma}
For each integer $n\geq1$, let $B_n$ be a separable real Hilbert space,
let $P_n$ be a nonempty Polish space, and let $j_n:P_n\to B_n$ be a
bounded continuous closed embedding. Suppose that
\begin{enumerate}
\item $H_n=\operatorname{span}_{\mathbb R}j_n(P_n)$ is $F_\sigma$ in $B_n$;
\item there is a continuous linear functional $\rho_n:B_n\to\mathbb R$
with $\rho_n(j_n(t))=1$ for all $t\in P_n$;
\item there is a countable family $\mathcal L_n\subseteq B_n^*$ such that
every $\ell\circ j_n$, for $\ell\in\mathcal L_n$, is locally constant,
and these functions separate the points of $P_n$.
\end{enumerate}
Choose $M_n\geq\sup_{t\in P_n}\|j_n(t)\|$ and put
$\beta_n=2^{-n}/(1+M_n)$. For a finite subset $F\subseteq\mathbb N_{>0}$,
let $B_F=\widehat\bigotimes_{n\in F}B_n$, with $B_\varnothing=\mathbb R$,
and form the Hilbert direct sum
\[
 \mathcal B=\bigoplus_{F\subseteq\mathbb N_{>0}\text{ finite}} B_F.
\]
Define $Q=\prod_{n\geq1}P_n$ and $J:Q\to\mathcal B$ by
\[
 J(q)_F=\left(\prod_{n\in F}\beta_n\right)
             \bigotimes_{n\in F}j_n(q_n).
\]
Then $J$ is a bounded continuous closed embedding, its real linear span
$\mathcal H=\operatorname{span}_{\mathbb R}J(Q)$ is $F_\sigma$, and $J$
again satisfies the mass-functional and locally constant separating
functional conditions in items \textup{(2)} and \textup{(3)}.
\end{lemma}

\begin{proof}
We give the closed cover of $\mathcal H$ explicitly. First observe that
\[
 \|J(q)\|^2
 =\prod_{n\geq1}\bigl(1+\beta_n^2\|j_n(q_n)\|^2\bigr)
 \leq\prod_{n\geq1}(1+4^{-n})<\infty.
\]
The sum defining the squared norm has tails tending to zero uniformly in $q$.
Explicitly, if $C_0=\prod_{n\geq1}(1+4^{-n})$, then
\[
 \sum_{F\not\subseteq\{1,\ldots,N\}}\|J(q)_F\|^2
 \leq C_0\left(\prod_{n>N}(1+4^{-n})-1\right)\longrightarrow0.
\]
Each individual tensor coordinate is continuous, so this uniform bound
also proves continuity of $J$. If $J(q^l)$ converges, its singleton
coordinate at $\{n\}$ shows that $j_n(q_n^l)$ converges. Since $j_n$ is
a closed embedding, this limit is $j_n(q_n)$ for a unique $q_n\in P_n$,
and $q_n^l\to q_n$. Thus $q^l\to q$ in $Q$, and continuity gives that
the original vector limit is $J(q)$. This proves that $J$ is a closed
embedding. The empty tensor coordinate is a mass functional taking
value $1$ on $J(Q)$. The functionals from item~\textup{(3)}, applied to
normalized singleton coordinates, are a countable locally constant
family separating the points of $Q$.

For a finite set $S\subseteq\mathbb N_{>0}$, set
\[
 E_S=\widehat\bigotimes_{n\in S}(\mathbb R\oplus B_n),\qquad
 A_S(q_S)=\bigotimes_{n\in S}(1,\beta_nj_n(q_n)),
\]
where $E_\varnothing=\mathbb R$ and $A_\varnothing=1$. The finite tensor
expansion identifies $E_S$ with $\bigoplus_{F\subseteq S}B_F$. Write
$\pi_S:\mathcal B\to E_S$ for the projection to this whole finite
packet of coordinates. Thus $\pi_S$ includes \emph{all} components
indexed by subsets of $S$, rather than only the component indexed by
$S$ itself. For disjoint finite sets, we use the corresponding canonical
tensor identifications without further notation.

For a fixed finite $B$, let $\mathcal B_{B^c}$ be the analogous direct
sum over finite subsets of $B^c$, and write $J_{B^c}$ for the same
construction using only these indices. There is a canonical Hilbert
space identification
\[
 \mathcal B=E_B\widehat\otimes\mathcal B_{B^c},\qquad
 J(q)=A_B(q_B)\otimes J_{B^c}(q_{B^c}).
\]
Put
\[
 V_B=\operatorname{span}_{\mathbb R}
       \{A_B(q_B):q_B\in\prod_{n\in B}P_n\}.
\]
This is an $F_\sigma$ subspace of $E_B$. To see this, the span of the
one-factor family $(1,\beta_nj_n(t))$ is exactly
\[
 V_{\{n\}}=\{(\rho_n(v),\beta_nv):v\in H_n\}.
\]
The map $v\mapsto(\rho_n(v),\beta_nv)$ is a closed linear embedding of
$B_n$ into $\mathbb R\oplus B_n$, so it sends a closed cover of $H_n$
to a closed cover of $V_{\{n\}}$. Finally,
$V_B=\bigotimes_{n\in B,\mathrm{alg}}V_{\{n\}}$, and
Lemma~\ref{lem:algebraic-tensor-fsigma} applies. Fix, for each finite
$B$, a closed cover $(S_{B,a})_{a\geq1}$ of $V_B$, and a countable
orthonormal basis of $E_B$. Let $\Lambda_B$ denote its coordinate
linear functionals; every nonzero element of $E_B$ has nonzero value
under some member of $\Lambda_B$.

Here is a countable collection of functionals that will provide finite
witnesses distinguishing tail labels. Each locally constant function
$\ell\circ j_n$ has countable range: its nonempty fibers are disjoint
open sets in a second countable space. Let $\mathcal A_n$ consist of
the constant-coordinate functional $(a,v)\mapsto a$ on
$\mathbb R\oplus B_n$, together with all functionals
\[
 (a,v)\longmapsto
 \frac{\beta_n^{-1}\ell(v)-b a}{c-b},
 \qquad \ell\in\mathcal L_n,\quad
 b,c\in(\ell\circ j_n)(P_n),\quad b\ne c.
\]
This is a countable collection. On the packet $(1,\beta_nj_n(t))$,
the latter functional takes value $1$ when $\ell(j_n(t))=c$ and value
$0$ when $\ell(j_n(t))=b$. For a finite $W$, let $\mathcal A_W$ be the
countable set of tensor-product functionals
$\bigotimes_{n\in W}\psi_n$ with $\psi_n\in\mathcal A_n$; for empty
$W$ it contains the identity functional on $\mathbb R$. Their values
on $A_W(q_W)$ are locally constant.

We now specify a closed chart. Its data consist of
\begin{enumerate}
\item a finite set $B$ and an integer $m\geq1$;
\item two disjoint finite sets $W_1,W_2\subseteq B^c$;
\item for $1\leq c\leq m$ and $r\in\{1,2\}$, a functional
$\phi_c^r\in\mathcal A_{W_r}$;
\item for $1\leq c\leq m$, an index $a_c$, a functional
$\lambda_c\in\Lambda_B$, and a common positive integer $k$.
\end{enumerate}
The chart consists of all vectors
\[
 y=\sum_{c=1}^m v_c\otimes J_{B^c}(q_c),
 \qquad v_c\in E_B,\quad q_c\in\prod_{n\notin B}P_n,
\]
whose representing parameters satisfy the closed coefficient conditions
\[
 v_c\in S_{B,a_c},\qquad \|v_c\|\leq k,
 \qquad |\lambda_c(v_c)|\geq k^{-1},
\]
and the finite pinning conditions
\begin{equation}
 \phi_c^r\bigl(A_{W_r}((q_d)_{W_r})\bigr)=\delta_{cd}
 \quad(1\leq c,d\leq m,\ r=1,2).
\label{eq:fock-pinning}
\end{equation}
Here $\delta_{cd}$ is the Kronecker delta. The pinning conditions are
closed, indeed clopen, conditions on the finitely many relevant
parameters. No requirement about eventual equality or inequality of
the tails is included in a chart. All the chart data have only
countably many choices.

Each chart is contained in $\mathcal H$, because $v_c\in V_B$ is a
finite linear combination of the finite packets $A_B$, and tensoring
such a packet with $J_{B^c}(q_c)$ gives a $J$-atom. We next prove that
each chart is closed in $\mathcal B$. Suppose that $y^l$ is a convergent
sequence of chart vectors, and choose a permitted representation
$(v_c^l,q_c^l)$ of each one. Projecting to $E_B\widehat\otimes E_{W_1}$
and contracting the second factor gives
\[
 (\operatorname{id}_{E_B}\otimes\phi_c^1)\pi_{B\cup W_1}(y^l)
 =v_c^l,
\]
by \eqref{eq:fock-pinning}. Therefore $v_c^l$ converges to some $v_c$.
The coefficient conditions are closed, so $v_c$ satisfies the same
conditions and in particular $|\lambda_c(v_c)|\geq k^{-1}$.

For $n\notin B\cup W_1$, project to the packet on
$B\cup W_1\cup\{n\}$, reorder its factors canonically, and contract
with $\lambda_c$ and $\phi_c^1$. This gives
\[
 (\lambda_c\otimes\phi_c^1\otimes\operatorname{id}_{\mathbb R\oplus B_n})
 \pi_{B\cup W_1\cup\{n\}}(y^l)
 =\lambda_c(v_c^l)(1,\beta_nj_n((q_c^l)_n)).
\]
The left side converges. Dividing by the nonzero scalar
$\lambda_c(v_c^l)$ shows that the one-coordinate packet on the right
converges. If $n\in W_1$, the identical argument uses $W_2$ and
$\phi_c^2$; this is possible because $W_1\cap W_2=\varnothing$. We
have thus proved convergence of $j_n((q_c^l)_n)$ for every $n\notin B$.
The closed-embedding hypothesis supplies $(q_c)_n\in P_n$ with
$(q_c^l)_n\to(q_c)_n$. Consequently $q_c^l\to q_c$ in the product.
The finite pinning conditions persist. By continuity of $J_{B^c}$ and
of the tensor product map,
\[
 y^l\longrightarrow\sum_{c=1}^m v_c\otimes J_{B^c}(q_c),
\]
which is again in the same chart. This proves its closedness. Notice
that tail labels may become eventually equal in this limit; the proof
does not prohibit this and does not require uniqueness of a finite
atom representation.

It remains to verify that these countably many charts, together with
$\{0\}$, cover $\mathcal H$. Take a nonzero finite linear combination
of $J$-atoms. Partition its finitely many labels $q\in Q$ according to
equality outside a finite set of coordinates. For each class choose
a representative. Choose a finite $B$ containing every coordinate
where a label differs from its class representative. The factorization
of $J$ then rewrites the given vector as
\[
 y=\sum_{c=1}^m v_c\otimes J_{B^c}(q_c),
 \qquad 0\ne v_c\in V_B,
\]
after discarding zero coefficients. Distinct remaining tail labels
$q_c,q_d$ differ at infinitely many coordinates outside $B$.

For each $r\in\{1,2\}$ and each ordered pair $c\ne d$, choose a
coordinate $n(c,d,r)\notin B$ where $(q_c)_n\ne(q_d)_n$, choosing all
these finitely many coordinates distinct. This is possible because
each pair differs infinitely often. Let $W_r$ be the set of coordinates
chosen for this value of $r$. At $n(c,d,r)$ choose a functional from
$\mathcal L_n$ distinguishing the two parameters, and then choose its
normalized member of $\mathcal A_n$ that takes value $1$ at $(q_c)_n$
and value $0$ at $(q_d)_n$. For fixed $c,r$, take the tensor product of
these features over $d\ne c$, and use the constant-coordinate feature
at unused coordinates of $W_r$. The resulting
$\phi_c^r\in\mathcal A_{W_r}$ satisfies
\eqref{eq:fock-pinning}: at $q_c$ every factor is $1$, and at each
other $q_d$ at least the factor indexed by $n(c,d,r)$ is $0$.
If $m=1$, both witness sets can be empty.

Finally choose closed cover members $S_{B,a_c}$ containing $v_c$,
coordinate functionals $\lambda_c$ nonzero on $v_c$, and a common
integer $k$ large enough for all coefficient bounds. These choices
place $y$ in one of the charts. Thus $\mathcal H$ is the union of a
countable family of closed sets, as required.
\end{proof}

\section{Families of arbitrarily high escape rank}\label{sec:transfinite}

\begin{definition}
An \emph{admissible family} consists of a nonempty Polish space $P$, a
separable real Hilbert space $B$, and a bounded continuous closed
embedding $j:P\to B$ such that
\begin{enumerate}
\item $H=\operatorname{span}_{\mathbb R}j(P)$ is $F_\sigma$;
\item some continuous linear functional $\rho\in B^*$ satisfies
$\rho(j(t))=1$ for all $t\in P$;
\item a countable family of continuous linear functionals on $B$ has
locally constant restrictions to $j(P)$ that separate its points.
\end{enumerate}
We say the family has \emph{diagonal escape rank at least $\alpha$} if
$\{(t,t):t\in P\}\subseteq R_{j,H}^\alpha$.
\end{definition}

The branch construction in Section~\ref{sec:construction} is an
admissible family of diagonal escape rank at least one. Denote it by
$(P_*,B_*,j_*,H_*,\rho_*)$ in this section. Indeed, its mass functional
is the $h_\varnothing$ coordinate. The linear functions
$v\mapsto\langle v,h_s\rangle/r_s$ restrict to the characteristic
functions of the cylinders $[s]$, hence are locally constant and
separate branches. Lemma~\ref{lem:closedtails}, applied at the root,
and its digit-stabilization proof show that $j_*$ is a closed embedding.
Finally, every neighborhood of a diagonal pair $(t,t)$ contains
$[s]\times[s]$ for some prefix $s$ of $t$. Proposition~\ref{prop:escape}
therefore puts every such pair in $R_{j_*,H_*}^1$.

\begin{lemma}[Successor step]\label{lem:successor-family}
If an admissible family $(P,B,j,H,\rho)$ has diagonal escape rank at
least $\alpha$, then there is an admissible family of diagonal escape
rank at least $\alpha+1$.
\end{lemma}

\begin{proof}
On $P\times P_*$ define
\[
 J(t,s)=(j(t),j_*(s))\in B\oplus B_*.
\]
Its algebraic span is exactly
\begin{equation}
 H' =\{(v,w)\in H\times H_*:\rho(v)=\rho_*(w)\}.
\label{eq:mass-matched-span}
\end{equation}
One inclusion follows by applying the mass functionals to a finite
linear combination. For the reverse, fix reference parameters $t_0,s_0$.
For any mass-one generating family one has
\[
 \operatorname{span}\{j(t)-j(t_0):t\in P\}=H\cap\ker\rho.
\]
Indeed, the mass of a finite combination is the sum of its coefficients.
The span of $J$ contains $(j(t)-j(t_0),0)$ and
$(0,j_*(s)-j_*(s_0))$, as differences of $J$-atoms. Subtracting
$\rho(v)J(t_0,s_0)$ from a mass-matched pair reduces both components
to these mass-zero spans, proving \eqref{eq:mass-matched-span}.

The set in \eqref{eq:mass-matched-span} is $F_\sigma$, being the
intersection of $H\times H_*$ with a closed linear condition. The
product family is bounded and is a continuous closed embedding.
Projection to the first factor followed by $\rho$ is a mass functional,
and the separating locally constant linear functions are inherited
from the two factors. Thus the family is admissible.

Fix any $s,r\in P_*$. The maps $t\mapsto(t,s)$ and $u\mapsto(u,r)$
send $R_{j,H}^\xi$ into $R_{J,H'}^\xi$ by
Lemma~\ref{lem:escape-tools}: their sum is
\[
 (j(t)+j(u),\ j_*(s)+j_*(r)),
\]
and an escaping limit in the first coordinate is outside $H'$. By
the rank assumption, for every $t\in P$ and every $s,r\in P_*$,
\begin{equation}
 ((t,s),(t,r))\in R_{J,H'}^\alpha.
\label{eq:successor-slice}
\end{equation}
Now fix $(t,s)$ and any neighborhood of its diagonal pair in
$(P\times P_*)^2$. Keep $t$ fixed and shrink the two second-coordinate
neighborhoods to a common neighborhood of $s$. Since
$(s,s)\in R_{j_*,H_*}^1$, there are pairs $(s_l,r_l)$ in this
neighborhood with $j_*(s_l)+j_*(r_l)\to z\notin H_*$. The pairs in
\eqref{eq:successor-slice} lie in the preceding remainder and their
$J$-sums converge to $(2j(t),z)\notin H'$. Hence the diagonal pair
survives to stage $\alpha+1$.
\end{proof}

\begin{lemma}[Limit step]\label{lem:limit-family}
Let $\lambda$ be a nonzero countable limit ordinal and let
$(\alpha_n)_{n\geq1}$ be an increasing sequence of ordinals less than
$\lambda$ with supremum $\lambda$. If for each $n$ there is an admissible
family of diagonal escape rank at least $\alpha_n$, then there is one
of diagonal escape rank at least $\lambda$.
\end{lemma}

\begin{proof}
Apply Lemma~\ref{lem:fock-fsigma} to the given families, with the
notation $J:Q\to\mathcal H\subseteq\mathcal B$ from that lemma.
It provides all admissibility conditions. Fix $q\in Q$ and an index
$i$. Let $w=J_{\{i\}^c}(q_{\{i\}^c})$ be the fixed outside tensor
family vector. In the canonical factorization
\[
 \mathcal B=(\mathbb R\oplus B_i)
                   \widehat\otimes\mathcal B_{\{i\}^c},
\]
define a bounded linear map
\[
 T_i:B_i\to\mathcal B,\qquad
 T_i(v)=(\rho_i(v),\beta_i v)\otimes w.
\]
For $t\in P_i$ this gives
\[
 T_i(j_i(t))=J(q\text{ with coordinate }i\text{ replaced by }t).
\]
Let $p_i:\mathcal B\to B_i$ be projection to the singleton coordinate
$\{i\}$ divided by $\beta_i$. Since the empty component of $w$ is $1$,
$p_iT_i$ is the identity. Moreover,
\[
 p_i(\mathcal H)\subseteq H_i,\qquad T_i(H_i)\subseteq\mathcal H.
\]
The first follows by applying $p_i$ to the generating atoms; the
second follows by applying $T_i$ to a finite linear combination of
$j_i$-atoms. In particular, $z\notin H_i$ implies
$T_i(z)\notin\mathcal H$.

Lemma~\ref{lem:escape-tools}, using the coordinate insertion in both
parameter variables, therefore sends $R_{j_i,H_i}^\xi$ into
$R_{J,\mathcal H}^\xi$. Applied at $(q_i,q_i)$ this gives
$(q,q)\in R_{J,\mathcal H}^{\alpha_i}$ for every $i$. The remainders
are decreasing and $(\alpha_i)$ is cofinal in $\lambda$, so
$(q,q)\in R_{J,\mathcal H}^\lambda$. The point $q$ was arbitrary.
\end{proof}

\begin{theorem}\label{thm:arbitrary-ranks}
For every countable ordinal $\alpha$ there is an admissible family
$(P_\alpha,B_\alpha,j_\alpha,H_\alpha,\rho_\alpha)$ of diagonal escape
rank at least $\alpha$. For $\alpha>0$ the Hilbert space $B_\alpha$
may be taken isometric to $\ell_2$.
\end{theorem}

\begin{proof}
At stage zero use a one-point parameter space, $B_0=\mathbb R$,
$j_0=1$, $H_0=\mathbb R$, and the identity mass functional. The
separation condition is vacuous. Apply
Lemma~\ref{lem:successor-family} at successors and
Lemma~\ref{lem:limit-family} at nonzero countable limits. This is a
transfinite recursion through the countable ordinals; each individual
limit uses only countably many earlier families. Finite tensor products
and countable Hilbert direct sums of separable Hilbert spaces are
separable, so every ambient space remains separable. Every positive
successor contains a copy of $B_*$ and is infinite dimensional; every
nonzero limit contains an infinite dimensional singleton summand from
an earlier positive stage. Thus each $B_\alpha$, $\alpha>0$, is a
separable infinite dimensional real Hilbert space and hence isometric
to $\ell_2$.
\end{proof}

\subsection{The finite stages explicitly}

The first few derivatives admit a concrete description. This also
illustrates the difference between a closed generating family and a
span with nontrivial escape behavior.

\begin{proposition}\label{prop:finite-stages}
For an integer $N\geq1$, let
\[
 J_N(x_1,\ldots,x_N)=(j_*(x_1),\ldots,j_*(x_N)),\qquad
 K_N=\operatorname{span}_{\mathbb R}J_N(P_*^N).
\]
For $0\leq k\leq N$, the $k$th escape remainder of this family is
\[
 R_{J_N,K_N}^k
 =\{(x,y)\in(P_*^N)^2:\#\{i:x_i=y_i\}\geq k\}.
\]
The next remainder $R_{J_N,K_N}^{N+1}$ is empty.
\end{proposition}

\begin{proof}
The mass-matching argument of Lemma~\ref{lem:successor-family} gives
\[
 K_N=\{(v_1,\ldots,v_N)\in H_*^N:
           \rho_*(v_1)=\cdots=\rho_*(v_N)\}.
\]
Two elementary closedness facts will be useful. Each $2j_*([s])$ is
closed. If $s,t$ are distinct prefixes of the same length, then
$j_*([s])+j_*([t])$ is also closed: after subtracting the fixed prefix
vectors $v_s+v_t$, the two summands lie in orthogonal closed descendant
coordinate subspaces, and each tail image is closed by
Lemma~\ref{lem:closedtails}.

Induct on $k$, starting with $k=0$. Suppose the formula holds at stage
$k$, and consider a pair with exactly $k$ equal block parameters.
On each equal block choose a common cylinder neighborhood of its two
parameters; on each unequal block choose disjoint cylinders with
prefixes of the same length. Within this product neighborhood, a pair
in $R^k$ must remain equal on all the originally equal blocks, since
the other blocks cannot become equal. Its sum image is therefore the
finite product of closed sets of the form $2j_*([s])$ on the equal
blocks and $j_*([s])+j_*([t])$ on the unequal blocks. This product is
closed. Every coordinate has mass $2$, so the product is contained in
$K_N$. The pair is removed at the next derivative.

Conversely, at a pair with at least $k+1$ equal blocks, keep $k$ of
these blocks equal and fixed. In any given neighborhood, perform the
paired escape from Proposition~\ref{prop:escape} in one further
equal block, keeping every other parameter fixed. All pairs remain
in $R^k$, while their sums converge to a vector with one coordinate
outside $H_*$. That limit is outside $K_N$, so the pair survives.
This proves the formula at the next stage. At stage $N$, every
remaining pair has exactly $N$ agreements and is removed by the
same closedness argument, giving $R^{N+1}=\varnothing$.
\end{proof}

\section{Nonexistence of a universal relation}\label{sec:conclusion}

\begin{proof}[Proof of Theorem~\ref{thm:main}]
Let $F$ be any $F_\sigma$ equivalence relation on a nonempty Polish
space $Y$. Fix a closed reflexive symmetric exhaustion as in
Section~\ref{sec:target}. Let $r_n$ be countable ordinals with
$Z_n^{r_n}=\varnothing$, and put $\gamma=\sup_n r_n$.
Reflexivity and nonemptiness of $Y$ imply $Z_n\ne\varnothing$, so
$\gamma\geq1$.

Take the family from Theorem~\ref{thm:arbitrary-ranks} for
$\alpha=\gamma$. Its span $H_\gamma$ is $F_\sigma$, so its coset
relation $E_{H_\gamma}$ is $F_\sigma$. If a Baire-measurable reduction
of this relation to $F$ existed, Corollary~\ref{cor:rank-obstruction}
would give an $n$ with $Z_n^\gamma\ne\varnothing$. But
$r_n\leq\gamma$ and the remainders are decreasing, so
$Z_n^\gamma\subseteq Z_n^{r_n}=\varnothing$, a contradiction.
Using a linear isometric identification of $B_\gamma$ with $\ell_2$ gives the claimed ambient space.
An empty target space cannot receive a reduction from a nonempty
space and is excluded immediately. Thus no $F_\sigma$ equivalence
relation can be universal for the class.
\end{proof}

\begin{corollary}\label{cor:bounded-difference}
The relation $E_{\ell_\infty}$ is not universal for $F_\sigma$
equivalence relations. The branch-space relation $E_{H_*}$ already
witnesses this failure.
\end{corollary}

\begin{proof}
For bounded difference use the closed levels
$F_n=\{(a,b):\sup_i|a(i)-b(i)|\leq n\}$. If
$(a,b,c)\in Z_n$, the triangle inequality gives $(a,c)\in F_{2n}$.
Thus $m_n\leq2n$ everywhere on $Z_n$, and $Z_n^1=\varnothing$.
Every diagonal pair of the branch family belongs to
$R_{j_*,H_*}^1$, contradicting Corollary~\ref{cor:rank-obstruction}
if a reduction existed.
\end{proof}

The same main theorem applies with $F=E_{H_*}$, or with any other
proposed $F_\sigma$ candidate. The obstruction is stronger than failure
of a particular construction: even the coset relations of $F_\sigma$
real linear subspaces of a fixed Hilbert space have no common
$F_\sigma$ upper bound under Borel reducibility.

\noindent\begin{minipage}{\textwidth}
\small
\textsc{School of Mathematical Sciences and School of Pre-university,}\\
\textsc{Dalian Minzu University}\\[0.4em]
\textit{Email address:}\enspace
\href{mailto:wangruijun@dlnu.edu.cn}{\texttt{wangruijun@dlnu.edu.cn}}
\end{minipage}

\end{document}